\documentclass[10pt]{amsart}

\usepackage{amsmath, mathtools, amsthm, amscd, amssymb, amsfonts, pstricks,bbm, verbatim, mathrsfs, tikz, subfig}
\usepackage[super]{nth}
\usepackage {tikz}
\usetikzlibrary {positioning}

\newcommand{\R}{{\mathbb R}}

\newtheorem{theorem}{Theorem}[section]

\newtheorem{corollary}[theorem]{Corollary}

\newtheorem{definition}[theorem]{Definition}

\numberwithin{equation}{section}

\theoremstyle{definition}

\newtheorem{remark}[theorem]{Remark}

\begin{document}
	
	\title[Spectral Properties Of The Gramian Of Finite Ultrametric Spaces]{Spectral Properties Of The Gramian Of Finite Ultrametric Spaces}
	\author[Gavin Robertson]{Gavin Robertson}
	\email{gr9625@gmail.com}
	
	\begin{abstract}
		The concept of $p$-negative type is such that a metric space $(X,d_{X})$ has $p$-negative type if and only if $(X,d_{X}^{p/2})$ embeds isometrically into a Hilbert space. If $X=\{x_{0},x_{1},\dots,x_{n}\}$ then the $p$-negative type of $X$ is intimately related to the Gramian matrix $G_{p}=(g_{ij})_{i,j=1}^{n}$ where $g_{ij}=\frac{1}{2}(d_{X}(x_{i},x_{0})^{p}+d_{X}(x_{j},x_{0})^{p}-d_{X}(x_{i},x_{j})^{p})$. In particular, $X$ has strict $p$-negative type if and only if $G_{p}$ is strictly positive semidefinite. As such, a natural measure of the degree of strictness of $p$-negative type that $X$ possesses is the minimum eigenvalue of the Gramian $\lambda_{min}(G_{p})$. In this article we compute the minimum eigenvalue of the Gramian of a finite ultrametric space. Namely, if $X$ is a finite ultrametric space with minimum nonzero distance $\alpha_{1}$ then we show that $\lambda_{min}(G_{p})=\alpha_{1}^{p}/2$. We also provide a description of the corresponding eigenspace.
	\end{abstract}
	\maketitle
	
	\section{Introduction}\label{Sec 1}
	
	The theory of $p$-negative type was originally introduced by Schoenberg \cite{Schoenberg 1, Schoenberg 2} to study isometric embeddings of finite metric spaces into Euclidean space. Its definition is as follows. The definition of strict $p$-negative type came later and is due to \cite{Li 1}.
	
	\begin{definition}
		Let $(X,d_{X})$ be a metric space and $p\geq0$. Then $X$ is said to have $p$-negative type if\footnote{For $p=0$ we use the convention that $0^{0}=0$.}
		\[ \sum_{i,j=1}^{n}d_{X}(x_{i},x_{j})^{p}\xi_{i}\xi_{j}\leq0
		\]
		for all distinct $x_{1},\dots,x_{n}\in X$, all $\xi_{1},\dots,\xi_{n}\in\R$ with $\sum_{i=1}^{n}\xi_{i}=0$ and $n\geq 2$. Further, $X$ is said to have strict $p$-negative type if equality holds only for $\xi_{1}=\dots=\xi_{n}=0$.
	\end{definition}
	
	Among other things, Schoenberg \cite{Schoenberg 1, Schoenberg 2} was able to prove the following results pertaining to $p$-negative type, which explains its connection to isometric embeddings into Hilbert space. (Again, the results for strict $p$-negative type are due to Li and Weston in \cite{Li 1}.)
	
	\begin{theorem}
		Let $(X,d_{X})$ be a metric space and $p\geq 0$.
		\begin{enumerate}
			\item $X$ has $p$-negative type if and only if $(X,d_{X}^{p/2})$ embeds isometrically into a Hilbert space.
			\item $X$ has strict $p$-negative type if and only if $(X,d_{X}^{p/2})$ embeds isometrically into a Hilbert space as an affinely independent set.
			\item If $0\leq q<p$ and $X$ has $p$-negative type, then $X$ also has strict $q$-negative type.
		\end{enumerate}
	\end{theorem}
	
	Note that by taking $p=2$ in the above theorem, one obtains the following characterisation of those metric spaces that embed into a Hilbert space: a metric space $(X,d_{X})$ embeds into a Hilbert space if and only if $X$ has $2$-negative type. Also, note that the third result in the above theorem implies that the $p$-negative type values of a metric space are all determined by the largest value of $p$ for which $X$ has $p$-negative type. If $(X,d_{X})$ is a metric space then we define the supremal $p$-negative type of $X$, denoted by $\wp_{X}$, to be
	\[ \wp_{X}=\sup\{p\geq0:X\text{ has $p$-negative type}\}\in[0,\infty].
	\]
	Since all of the sums that appear in the definition of $p$-negative type are finite sums a simple limiting argument shows that if $\wp_{X}<\infty$ then $X$ has $\wp_{X}$-negative type. Consequently, the set of values $p\geq 0$ such that a metric space has $p$-negative type is always of the form $[0,\wp_{X}]$ or $[0,\infty)$.
	
	To quantify the degree of strictness of $p$-negative type that a metric space possesses, the $p$-negative type gap was introduced by Doust and Weston in \cite{Doust 1, Doust 3}. If $(X,d_{X})$ is a metric space and $p\geq 0$ then the $p$-negative type gap of $X$ is the largest non-negative constant $\Gamma=\Gamma_{X}(p)$ such that
	\[ \frac{\Gamma}{2}\bigg(\sum_{k=1}^{n}|\xi_{k}|\bigg)^{2}+\sum_{i,j=1}^{n}d_{X}(x_{i},x_{j})^{p}\xi_{i}\xi_{j}\leq0
	\]
	for all distinct $x_{1},\dots,x_{n}\in X$, $\xi_{1},\dots,\xi_{n}\in\R$ with $\sum_{i=1}^{n}\xi_{i}=0$ and $n\geq 2$.
	
	Most importantly, if $X$ is a finite metric space and $p\geq 0$ then $X$ has strict $p$-negative type if and only if $\Gamma_{X}(p)>0$ (see \cite{Li 1}). The usefulness of knowing the value of $\Gamma_{X}(p)$ when $X$ has strict $p$-negative type is seen in the following result of \cite{Li 1}. If $X$ is a finite metric space with $|X|=n\geq 3$ and $X$ has strict $p$-negative type, then $X$ also has strict $q$-negative type for all $q\in[p,p+\epsilon)$ where
	\[ \epsilon=\frac{\ln\big(1+\frac{\Gamma_{X}(p)}{D_{X}^{p}\gamma(n)}\big)}{\ln\mathcal{D}_{X}}
	\]
	with $D_{X}=\max_{x,y\in X}d_{X}(x,y)$, $\mathcal{D}_{X}=D_{X}/\min_{x,y\in X,x\neq y}d_{X}(x,y)$ and $\gamma(n)=1-\frac{1}{2}(\lfloor\frac{n}{2}\rfloor^{-1}+\lceil\frac{n}{2}\rceil^{-1})$. That is, if $\Gamma_{X}(p)>0$ is known then the values of $q$ for which $X$ has $q$-negative type can be extended to certain values of $q>p$.
	
	Doust and Weston in \cite{Doust 1} were able to provide an explicit formula for the $1$-negative type gap of an arbitrary weighted finite metric tree. Using this, they were able to provide non-trivial lower bounds on the supremal $p$-negative type of finite metric trees. Li and Weston in \cite{Li 1} were able to provide an explicit formula for the $0$-negative type gap of a finite metric space, which led them to conclude non-trivial bounds on the supremal $p$-negative type of an $n$ point finite metric space. For finite ultrametric spaces, the authors in \cite{Doust 2} were able to provide an elegant combinatorial formula for the limit of $\Gamma_{X}(p)$ as $p\rightarrow\infty$. In \cite{Wolf 1}, Wolf provided a versatile formula for computing the $p$-negative type gap of a finite metric space, and among other things, was able to provide a computation of the $1$-negative type gap for the cycle $C_{n}$, when $n$ is odd. Also, in \cite{Wolf 2}, Wolf showed that the $p$-negative type gap is closely related to certain eigenvalues of the distance matrix of $(X,d_{X}^{p})$.
	
	Recently, many $p$-negative type properties of finite metric spaces, such as the $p$-negative type gap, have been studied through the lens of the Gramian matrix, whose definition is as follows.
	
	\begin{definition}
		Let $(X,d_{X})=(\{x_{0},x_{1},\dots,x_{n}\},d_{X})$ be a metric space and $p\geq 0$. The \textit{$p$-Gramian} of $X$ is the matrix $G_{p}(X)=(g_{ij})_{i,j=1}^{n}$ where
		\[ g_{ij}=\frac{1}{2}(d_{X}(x_{i},x_{0})^{p}+d_{X}(x_{j},x_{0})^{p}-d_{X}(x_{i},x_{j})^{p})
		\]
		for all $1\leq i,j\leq n$.
	\end{definition}

	\begin{remark}
		Note that the Gramian of the metric space depends not only on the metric structure of the space but also on the labelling of the points $x_{0},x_{1},\dots,x_{n}$.
	\end{remark}
	
	It follows from Schoenberg's original arguments \cite{Schoenberg 1, Schoenberg 2} that a finite metric space has (strict) $p$-negative type if and only if it is (strictly) positive semidefinite. Thus another natural measure of the strictness of $p$-negative type of a finite metric space is the minimum eigenvalue of the Gramian matrix $\lambda_{min}(G_{p})$. In particular, it is obvious that a finite metric space has strict $p$-negative type if and only if $\lambda_{min}(G_{p})>0$. In \cite{Robertson 2} the relation between $\lambda_{min}(G_{p})$ and the usual $p$-negative type gap as stated above is given. The details of this relation are expanded on in Section \ref{Gramian Section}.
	
	In \cite{Robertson 2} the minimum eigenvalue of the Gramian $\lambda_{min}(G_{p})$ was computed for the complete bipartite graphs $K_{n,1}$. Apart from this however, there are no other spaces for which $\lambda_{min}(G_{p})$ has been computed. The purpose of this article is to compute $\lambda_{min}(G_{p})$ for the class of spaces known as ultrametric spaces, and also to describe the corresponding eigenspace. Most notably here we are able to compute $\lambda_{min}(G_{p})$ for all $p\geq 0$, whereas, as mentioned above, only the behaviour of $\Gamma_{X}(p)$ as $p\rightarrow\infty$ is known.
	
	In Section \ref{Gramian Section} the relevant background on $\lambda_{min}(G_{p})$ and its relation to the usual $p$-negative type gap is given. Then in Section \ref{Ultrametric Section} the basic definitions pertaining to ultrametric spaces are covered. In Section \ref{Minimum Eigenvalue Section} we show that if $X$ is a finite ultrametric space with minimum nonzero distance $\alpha_{1}$ then $\lambda_{min}(G_{p})=\alpha_{1}^{p}/2$. Then in Section \ref{Eigenspace Section} we provide a description of the corresponding eigenspace. Finally, in Section \ref{Example Section} we provide an explicit example of these concepts for a specific finite ultrametric space.

	\section{Eigenvalues of The Gramian and the $p$-negative Type Gap}\label{Gramian Section}

	In \cite{Robertson 2} a more general $p$-negative type gap than that of Doust and Weston in \cite{Doust 1, Doust 3} was introduced. For this we require the following definitions.

	\begin{definition}
		By a \textit{weight vector} in $\R^{2n+2}$, we mean a vector consisting of nonnegative real numbers
		\[ (s,t)=(s_{0},s_{1},\dots,s_{n},t_{0},t_{1},\dots,t_{n})\in\R^{2n+2}
		\]
		such that the following conditions hold.
		\begin{enumerate}
			\item $\sum_{i=0}^{n}s_{i}=\sum_{i=0}^{n}t_{i}$.
			\item $\sum_{i=0}^{n}s_{i}t_{i}=0$ (i.e. $s,t$ are disjointly supported).
		\end{enumerate}
		The set of all weight vectors in $\R^{n+2}$ will be denoted by $W_{n}$.
	\end{definition}

	\begin{definition}
		A nonempty subset $S\subseteq W_{n}$ will be called an \textit{allowable set} if it satisfies the following three properties.
		\begin{enumerate}
			\item $0\notin S$.
			\item For all $x\in W_{n}$ with $x\neq 0$, there exists some $\lambda>0$ such that $\lambda x\in S$.
			\item $S$ is compact (as a subset of $\R^{2n+2}$ in its usual topology).
		\end{enumerate}
	\end{definition}
	
	We may now define our more general $p$-negative type gap functions.
	
	\begin{definition}
		Let $(X,d_{X})=(\{x_{0},x_{1},\dots,x_{n}\},d_{X})$ be a metric space and let $S\subseteq W_{n}$ be an allowable set. The \textit{$p$-negative type gap function} (corresponding to $S$ and $X$) is the function $\Gamma_{S}:[0,\infty)\rightarrow\R$ defined by
		\[ \Gamma_{S}(p)=\frac{1}{2}\bigg\{\inf_{(s,t)\in S}2\sum_{i,j=0}^{n}s_{i}t_{j}d_{X}(x_{i},x_{j})^{p}-\sum_{i,j=0}^{n}(s_{i}s_{j}+t_{i}t_{j})d_{X}(x_{i},x_{j})^{p}\bigg\}
		\]
		for all $p\geq 0$.
	\end{definition}

	It was shown in \cite{Robertson 2} that for a certain choice of $S$, one actually recovers the minimum eigenvalue of the Gramian matrix.
	
	\begin{corollary}\label{Min Eigenvalue Formula}
		Let $(X,d_{X})=(\{x_{0},x_{1},\dots,x_{n}\},d_{X})$ be a metric space, $p\geq 0$ and $G=G_{p}(X)$ the $p$-Gramian of $X$. Also, set
		\begin{align*}
			S&=\bigg\{(s,t)\in W_{n}:\sum_{i=1}^{n}s_{i}^{2}+\sum_{i=1}^{n}t_{i}^{2}=1\bigg\}.
		\end{align*}
		Then
		\[ \Gamma_{S}(p)=\lambda_{min}(G).
		\]
	\end{corollary}

	\section{Finite Ultrametric Spaces}\label{Ultrametric Section}
	
	Ultrametric spaces are an important family of metric spaces that have recently featured in a variety of different applications in the theory of metric embeddings. Let us recall their definition now.
	
	\begin{definition}
		A metric space $(X,d_{X})$ is said to be an ultrametric space if $d_{X}(x,y)\leq\max\{d_{X}(x,z),d_{X}(y,z)\}$ for all $x,y,z\in X$.
	\end{definition}

	\begin{remark}
		To avoid trivial cases, we shall always assume that $|X|\geq 3$.
	\end{remark}

	An important notion in the theory of finite ultrametric spaces is the following one.
	
	\begin{definition}
		Let $(X,d_{X})$, $|X|>1$, be a finite ultrametric space with minimum nonzero distance $\alpha$. Let $z\in X$ be given. The set $B_{z}(\alpha)=\{x\in X:d_{X}(x,z)\leq \alpha\}$ is said to be a coterie in $X$ if $|B_{z}(\alpha)|>1$.
	\end{definition}

	Some simple facts about coteries are the following. If $(X,d_{X})$ is a finite ultrametric space on at least two points then $X$ contains at least one coterie. Morever, if $B_{1}$ and $B_{2}$ are coteries then either $B_{1}=B_{2}$ or $B_{1}\cap B_{2}=\emptyset$.

	Ultrametric spaces are particularly important in the theory of $p$-negative type since they form the class of spaces which have infinite supremal $p$-negative type. That is, it was shown in \cite{Faver 1} that a metric space $(X,d_{X})$ has $\wp_{X}=\infty$ if and only if $X$ is an ultrametric space. As such, if $X$ is a finite ultrametric space then $\Gamma_{X}(p)>0$ for all $p>0$. Because of this, work has recently been done in the direction of calculating an explicit expression for $\Gamma_{X}(p)$ for all $p>0$ (see \cite{Wolf 2, Doust 2, Faver 1}). In particular, a formula for the limit of $\Gamma_{X}(p)$ as $p\rightarrow\infty$ was given for finite ultrametric spaces in \cite{Doust 2}.
	
	\section{The Minimum Eigenvalue Of The Gramian}\label{Minimum Eigenvalue Section}
	
	Throughout this section we let $(X,d_{X})=(\{x_{0},x_{1},\dots,x_{n}\},d_{X})$ be a finite ultrametric space with nonzero distances $\alpha_{1}<\dots<\alpha_{\ell}$. We also assume that $S$ is an allowable set and that $(s,t)\in S$.
	
	For $p>0$ let
		\[ \gamma(p)=2\sum_{i,j=0}^{n}s_{i}t_{j}d_{X}(x_{i},x_{j})^{p}-\sum_{i,j=0}^{n}(s_{i}s_{j}+t_{i}t_{j})d_{X}(x_{i},x_{j})^{p}.
		\]
	Then since the non-zero distances in $(X,d_{X})$ are $\alpha_{1},\dots,\alpha_{\ell}$ we have that $\gamma(p)=c_{1}\alpha_{1}^{p}+\dots+c_{\ell}\alpha_{\ell}^{p}$ for some $c_{1},\dots,c_{\ell}\in\R$.
	
	In \cite{Doust 2} an explicit expression was given for the constants $c_{i}$ in terms of $s_{i}$ and $t_{j}$ and the structure of the space $X$. For our purposes we require only the following properties of the constants $c_{i}$, which were proved in \cite{Doust 2}.
	
	\begin{theorem}\label{Constants Properties}
		The constants $c_{i}$ satisfy the following properties.
		\begin{enumerate}
			\item $c_{k}+c_{k+1}+\dots+c_{\ell}\geq 0$ for all $1\leq k\leq\ell$.
			\item $c_{1}+\dots+c_{\ell}=\sum_{i=0}^{n}s_{i}^{2}+\sum_{i=0}^{n}t_{i}^{2}$.
		\end{enumerate}
	\end{theorem}

	\begin{remark}
		In \cite{Doust 2} it was the case that $2(c_{1}+\dots+c_{\ell})=\sum_{i=0}^{n}s_{i}^{2}+\sum_{i=0}^{n}t_{i}^{2}$. The reason for this difference is that our $\gamma(p)$ is actually a factor of $2$ times the definition of $\gamma(p)$ as given in \cite{Doust 2}.
	\end{remark}
	
	We may use these properties of the constants $c_{i}$ to compute $\lambda_{min}(G_{p})$.
	
	\begin{theorem}\label{Minimum Gamma}
		Let
			\[ S=\bigg\{(s,t)\in W_{n}:\sum_{i=1}^{n}s_{i}^{2}+\sum_{i=1}^{n}t_{i}^{2}=1\bigg\}.
			\]
		Then $\gamma(p)\geq\alpha_{1}^{p}$ for all $(s,t)\in S$ with equality if and only if $c_{2}=\dots=c_{\ell}=0$ and $s_{0}=t_{0}=0$.
	\end{theorem}
	\begin{proof}
		Here we use the properties of the constants $c_{i}$ from Theorem \ref{Constants Properties} and the fact that $0<\alpha_{1}<\dots<\alpha_{\ell}$. So, since $c_{\ell}\geq0$ and $\alpha_{\ell-1}<\alpha_{\ell}$ we have that
			\begin{align*}
				\gamma(p)&=c_{1}\alpha_{1}^{p}+\dots+c_{\ell}\alpha_{\ell}^{p}
				\\&\geq c_{1}\alpha_{1}^{p}+\dots+c_{\ell-1}\alpha_{\ell-1}^{p}+c_{\ell}\alpha_{\ell-1}^{p}
				\\&=c_{1}\alpha_{1}^{p}+\dots+(c_{\ell-1}+c_{\ell})\alpha_{\ell-1}^{p}.
			\end{align*}
		Since $c_{\ell-1}+c_{\ell}\geq0$ and $\alpha_{\ell-2}<\alpha_{\ell-1}$
			\begin{align*}
				c_{1}\alpha_{1}^{p}+\dots+(c_{\ell-1}+c_{\ell})\alpha_{\ell-1}^{p}&\geq c_{1}\alpha_{1}^{p}+\dots+c_{\ell-2}\alpha_{\ell-2}^{p}+(c_{\ell-1}+c_{\ell})\alpha_{\ell-2}^{p}
				\\&=c_{1}\alpha_{1}^{p}+\dots+(c_{\ell-2}+c_{\ell-1}+c_{\ell})\alpha_{\ell-2}^{p}.
			\end{align*}
		Continuing in this manner one has that
			\[ \gamma(p)\geq(c_{1}+\dots+c_{\ell})\alpha_{1}^{p}=\bigg(\sum_{i=0}^{n}s_{i}^{2}+\sum_{i=0}^{n}t_{i}^{2}\bigg)\alpha_{1}^{p}=(1+s_{0}^{2}+t_{0}^{2})\alpha_{1}^{p}\geq\alpha_{1}^{p}.
			\]
		By the above string of inequalities, since $\alpha_{1}<\dots<\alpha_{\ell}$ it is clear that one has equality if and only if $c_{\ell}=c_{\ell-1}+c_{\ell}=\dots=c_{2}+\dots+c_{\ell}=0$ and $s_{0}=t_{0}=0$. But this happens if and only if $c_{2}=\dots=c_{\ell}=0$  and $s_{0}=t_{0}=0$ and so we are done.
	\end{proof}

	As long as one is careful in choosing $x_{0}$ we may now compute the minimum eigenvalue of the Gramian of a finite ultrametric space.
	
	\begin{definition}
		Let $(X,d_{X})=(\{x_{0},x_{1},\dots,x_{n}\},d_{X})$ be a finite ultrametric space. We will say that $X$ is degenerate if there is exactly one coterie $B\subseteq X$ and $|B|=2$ and $x_{0}\in B$. Otherwise we will say that $X$ is nondegenerate.
	\end{definition}

	\begin{remark}
		Note that since we are assuming that $|X|\geq 3$, the points in $X$ may always be reordered so that $X$ is nondegenerate.
	\end{remark}
	
	\begin{corollary}
		Let $(X,d_{X})=(\{x_{0},x_{1},\dots,x_{n}\},d_{X})$ be a nondegenerate finite ultrametric space. Then for any $p>0$ one has that $\lambda_{min}(G_{p})=\alpha_{1}^{p}/2$.
	\end{corollary}
	\begin{proof}
		By Corollary \ref{Min Eigenvalue Formula}, if one sets
			\[ S=\bigg\{(s,t)\in W_{n}:\sum_{i=1}^{n}s_{i}^{2}+\sum_{i=1}^{n}t_{i}^{2}=1\bigg\}
			\]
		then $\lambda_{min}(G_{p})=\Gamma_{S}(p)$. Hence
			\[ \lambda_{min}(G_{p})=\Gamma_{S}(p)=\frac{1}{2}\inf_{(s,t)\in S}\gamma(p)\geq\alpha_{1}^{p}/2.
			\]
		Now, since $|X|\geq 3$ and $X$ is nondegenerate there exists a coterie $B\subseteq X$ and $x_{i},x_{j}\in B$ with $1\leq i\neq j\leq n$. Then let $s=(s_{0},s_{1},\dots,s_{n})$ where $s_{i}=1/\sqrt{2}$ and $s_{k}=0$ otherwise, and let $t=(t_{0},t_{1},\dots,t_{n})$ where $t_{j}=1/\sqrt{2}$ and $t_{k}=0$ otherwise. Then clearly $(s,t)\in S$. Also, for this choice of $(s,t)$ it is straightforward to check that $\gamma(p)=\alpha_{1}^{p}$. Hence $\lambda_{min}(G_{p})=\alpha_{1}^{p}/2$ as required.
	\end{proof}
	
	\section{The Corresponding Eigenspace}\label{Eigenspace Section}
	
	In this section we provide a description of the eigenspace of the Gramian matrix corresponding to the minimum eigenvalue $\lambda_{min}(G_{p})$. To do so we require the following characterisation of those $(s,t)\in S$ for which $c_{2}=\dots=c_{\ell}=0$ (see \cite{Doust 2}).
	
	Denote the distinct coteries of $X$ by $B_{1},\dots,B_{r}$ and let $X_{0}=X\setminus\bigcup_{i=1}^{r}B_{i}$. Also, for $Y\subseteq X$ we will let $\mathcal{I}(Y)=\{0\leq i\leq n:x_{i}\in Y\}$.
	
	Also, throughout this section we will assume that $X$ is nondegenerate.
	
	\begin{theorem}\label{Flat simplicies}
		Let $(s,t)\in S$. Then the following are equivalent.
		\begin{enumerate}
			\item $c_{2}=\dots=c_{\ell}=0$.
			\item $s_{i}=t_{i}=0$ for all $i\in\mathcal{I}(X_{0})$ and $\sum_{i\in\mathcal{I}(B_{j})}s_{i}=\sum_{i\in\mathcal{I}(B_{j})}t_{i}$ for all $1\leq j\leq r$.
		\end{enumerate}
	\end{theorem}
	
	Let
		\[ \Pi_{0}=\{\xi=(\xi_{0},\xi_{1},\dots,\xi_{n})^{T}\in\R^{n+1}:\xi_{0}+\xi_{1}+\dots+\xi_{n}=0\}
		\]
	and then put
		\[ F=\bigg\{\xi\in\Pi_{0}:\xi_{i}=0,\forall i\in\mathcal{I}(X_{0})\cup\{0\}\text{ and }\sum_{i\in \mathcal{I}(B_{j})}\xi_{i}=0,\forall 1\leq j\leq r\bigg\}.
		\]
	Also, define $\pi:\Pi_{0}\rightarrow\R^{n}$ by $\pi(\xi_{0},\xi_{1},\dots,\xi_{n})=(\xi_{1},\dots,\xi_{n})$.
	
	\begin{theorem}
		Let $E$ be the eigenspace of $G_{p}$ corresponding to the eigenvalue $\lambda_{min}(G_{p})=\alpha_{1}^{p}/2$. Then $E=\pi(F)$.
	\end{theorem}
	\begin{proof}
		Let $\lambda=\lambda_{min}(G_{p})=\alpha_{1}^{p}/2$. First suppose that $\eta\in E$. If $\eta=0$ then trivially $\eta\in\pi(F)$. Now suppose that $\eta\neq 0$. Since $\eta\in E$ one has that $\langle G_{p}\eta,\eta\rangle=\lambda\|\eta\|_{2}^{2}$. If one puts $\xi=\pi^{-1}(\eta)$ then simple manipulation of sums gives that $\langle G_{p}\eta,\eta\rangle=-\frac{1}{2}\sum_{i,j=0}^{n}d_{X}(x_{i},x_{j})^{p}\xi_{i}\xi_{j}$, and hence we have that
			\[ \sum_{i,j=0}^{n}d_{X}(x_{i},x_{j})^{p}\xi_{i}\xi_{j}+2\lambda\|\eta\|_{2}^{2}=0.
			\]
		Now define $s=(s_{0},s_{1},\dots,s_{n})^{T}\in\R^{n+1}$ by $s_{i}=\xi_{i}/\|\eta\|_{2}$ if $\xi_{i}\geq 0$ and $s_{i}=0$ otherwise. Similarly, define $t=(t_{0},t_{1},\dots,t_{n})\in\R^{n+1}$ by $t_{i}=-\xi_{i}/\|\eta\|_{2}$ if $\xi_{i}\leq 0$ and $t_{i}=0$ otherwise. Note in particular that $(s,t)\in W_{n}$ with $\sum_{i=1}^{n}s_{i}^{2}+\sum_{i=1}^{n}t_{i}^{2}=\|\eta\|_{2}^{2}/\|\eta\|_{2}^{2}=1$. Now, the above equality can be rearranged to read as
			\[ \gamma(p)=2\sum_{i,j=0}^{n}s_{i}t_{j}d_{X}(x_{i},x_{j})^{p}-\sum_{i,j=0}^{n}(s_{i}s_{j}+t_{i}t_{j})d_{X}(x_{i},x_{j})^{p}=\alpha_{1}^{p}.
			\]
		Hence, by Theorem \ref{Minimum Gamma} one has that $c_{2}=\dots=c_{\ell}=0$ and $s_{0}=t_{0}=0$ for the weight vector $(s,t)$. By Theorem \ref{Flat simplicies} this means that $s_{i}=t_{i}=0$ for all $i\in\mathcal{I}(X_{0})$ and $\sum_{i\in\mathcal{I}(B_{j})}s_{i}=\sum_{i\in\mathcal{I}(B_{j})}t_{i}$ for all $1\leq j\leq r$. By the definition of $(s,t)$ this is clearly equivalent to saying that $\xi\in F$. Hence $\eta=\pi(\xi)\in\pi(F)$.
		
		Conversely, suppose that $\eta\in\pi(F)$. Again, if $\eta=0$ then $\eta\in E$ trivially. Now suppose that $\eta\neq 0$. For this direction we simply reverse the above argument. So, $\xi=\pi^{-1}(\eta)\in F$. Thus if we define $s,t\in\R^{n+1}$ as above then the definition of $F$ implies that the weight vector $(s,t)$ has $\sum_{i=1}^{n}s_{i}^{2}+\sum_{i=1}^{n}t_{i}^{2}=\|\eta\|_{2}^{2}/\|\eta\|_{2}^{2}=1$ and $c_{2}=\dots=c_{\ell}=0$ and $s_{0}=t_{0}=0$. Hence one has that
			\[ \gamma(p)=2\sum_{i,j=0}^{n}s_{i}t_{j}d_{X}(x_{i},x_{j})^{p}-\sum_{i,j=0}^{n}(s_{i}s_{j}+t_{i}t_{j})d_{X}(x_{i},x_{j})^{p}=\alpha_{1}^{p}
			\]
		which can be rearranged to give
			\[ \langle G_{p}\eta,\eta\rangle=\lambda\|\eta\|_{2}^{2}.
			\]
		Since $\lambda$ is the minimum eigenvalue of the symmetric matrix $G_{p}$ this is enough to conclude that $\eta$ is an eigenvector with eigenvalue $\lambda$. That is, $\eta\in E$ and so we are done.
	\end{proof}
	
	\begin{corollary}\label{Dimension}
		If $x_{0}\in X_{0}$ then the dimension of the eigenspace $E$ is $\sum_{i=1}^{r}|B_{i}|-r$. Otherwise, if $x_{0}\in\bigcup_{i=1}^{r}B_{i}$ then the dimension of the eigenspace $E$ is $\sum_{i=1}^{r}|B_{i}|-r-1$.
	\end{corollary}
	\begin{proof}
		It is a simple matter to write down a basis for $F$. To this end, let $e_{0},e_{1},\dots,e_{n}$ be the standard basis for $\R^{n+1}$. Also, for each $1\leq j\leq r$ let $k_{j}\in\mathcal{I}(B_{j})$ with $k_{j}\neq 0$. Then it is not hard to see that the set
			\[ \{e_{i}-e_{k_{j}}:i\in\mathcal{I}(B_{j}),i\neq k_{j},i\neq 0,1\leq j\leq r\}
			\]
		is a basis for $F$. Hence the dimension of $F$ is given by $\sum_{i=1}^{r}|B_{i}|-r$ if $x_{0}\in X_{0}$, and $\sum_{i=1}^{r}|B_{i}|-r-1$ otherwise. Then, noting that $\pi:\Pi_{0}\rightarrow\R^{n}$ is a linear isomorphism it follows that $E$ has the same dimension as $F$.
	\end{proof}
	
	\section{An Example}\label{Example Section}
	
	Let $(X,d_{X})=(\{x_{0},x_{1},x_{2},x_{3},x_{4},x_{5},x_{6}\},d_{X})$ be the $7$ point metric space with distance matrix $D=(d_{X}(x_{i},x_{j}))_{i,j=0}^{6}$ given by
		\[ \begin{pmatrix} 0 & 3 & 3 & 4 & 4 & 4 & 4 \\ 3 & 0 & 1 & 4 & 4 & 4 & 4 \\ 3 & 1 & 0 & 4 & 4 & 4 & 4 \\ 4 & 4 & 4 & 0 & 1 & 2 & 2 \\ 4 & 4 & 4 & 1 & 0 & 2 & 2 \\ 4 & 4 & 4 & 2 & 2 & 0 & 1 \\ 4 & 4 & 4 & 2 & 2 & 1 & 0 \end{pmatrix}.
		\]
	It is tedious, yet not difficult, to check that $X$ is in fact an ultrametric space and that the coteries of $X$ are given by $B_{1}=\{x_{1},x_{2}\}$, $B_{2}=\{x_{3},x_{4}\}$ and $B_{3}=\{x_{5},x_{6}\}$. Thus $x_{0}\in X_{0}$ and in particular this implies that $x_{0}$ is nondegenerate.
	
	Note that the minimum nonzero distance in $X$ is $\alpha_{1}=1$ and hence the minimum eigenvalue of the Gramian $G_{p}$ is given by $\lambda_{min}(G_{p})=1/2$, for all $p\geq 0$. Keeping the notation from the proof of Corollary \ref{Dimension} we may choose $k_{1}=1, k_{2}=3, k_{3}=5$. Then a basis for $E$, the eigenspace of $G_{p}$ corresponding the the eigenvalue $1/2$, is given by
		\[ \{e_{2}-e_{1},e_{4}-e_{3},e_{6}-e_{5}\}=\{(-1,1,0,0,0,0)^{T},(0,0,-1,1,0,0)^{T},(0,0,0,0,-1,1)^{T}\}.
		\]

	\section*{Acknowledgments}
	
	
	\bibliographystyle{amsalpha}

\end{document}